\documentclass[10pt,a4paper]{amsart}

\theoremstyle{plain}
\newtheorem{theorem}{Theorem}[section]
\newtheorem{lemma}[theorem]{Lemma}

\newtheorem{corollary}[theorem]{Corollary}

\theoremstyle{definition}

\theoremstyle{remark}
\newtheorem{remark}[theorem]{Remark}

\def\bin #1#2 {\left( \matrix { #1 \cr #2 \cr } \right) }

\begin{document}

\title[A remark on the genus of curves in $\mathbf P^4$]
{A remark on the genus of curves in $\mathbf P^4$}

\author{Vincenzo Di Gennaro }
\address{Universit\`a di Roma \lq\lq Tor Vergata\rq\rq, Dipartimento di Matematica,
Via della Ricerca Scientifica, 00133 Roma, Italy.}
\email{digennar@axp.mat.uniroma2.it}

\bigskip
\abstract Let $C$ be an irreducible, reduced, non-degenerate curve,
of arithmetic genus $g$ and degree $d$, in the projective space
$\mathbf P^4$ over the complex field. Assume that $C$ satisfies the
following {\it flag condition of type $(s,t)$}: {$C$ does not lie on
any surface of degree $<s$, and on any hypersurface of degree $<t$}.
Improving previous results, in the present paper we exhibit a
Castelnuovo-Halphen type bound for $g$, under the assumption $s\leq
t^2-t$ and $d\gg t$. In the range $t^2-2t+3\leq s\leq t^2-t$, $d\gg
t$, we are able to give some information on the extremal curves.
They are arithmetically Cohen-Macaulay curves, and lie on a flag
like $S\subset F$, where $S$ is a surface of degree $s$, $F$ a
hypersurface of degree $t$, $S$ is unique, and its general
hyperplane section  is a space extremal curve, not contained in any
surface of degree $<t$. In the case $d\equiv 0$ (modulo $s$), they
are exactly the complete intersections of a surface $S$ as above,
with a hypersurface. As a consequence of previous results, we get a
bound for the speciality index of a curve satisfying a flag
condition.



\endabstract

\maketitle

\section{Introduction}

Let $C$ be an irreducible, reduced, non-degenerate curve, of
arithmetic genus $g$ and degree $d$, in the projective space
$\mathbf P^4$ over the complex field. Assume that $C$ satisfies the
following {\it  flag condition of type $(s,t)$}: {\it $C$ does not
lie on any surface of degree $<s$, and on any hypersurface of degree
$<t$}. Under the assumption $s>t^2-t$ and $d>\max(12(s+1)^2, s^3)$,
in \cite[Theorem]{P4}, one proves a sharp upper bound $G(d,s,t)$ for
$g$ (for the definition of $G(d,s,t)$, see Section 2, (ii), below).
In the present paper, we prove that this bound $G(d,s,t)$ applies
also when $s\leq t^2-t$ and $d\gg t$. More precisely, we prove the
following:

\begin{theorem}\label{main}
Let $C$ be an irreducible, reduced, non-degenerate curve, of
arithmetic genus $g$ and degree $d$, in the projective space
$\mathbf P^4$ over the complex field. Assume $C$ is not contained in
any hypersurface of degree $<t$ ($t\geq 3$), and in any surface of
degree $<s$ ($s\geq 3$). Define $\alpha$, $\beta$, $m$ and
$\epsilon$ by dividing $s-1=\alpha t+\beta$, $0\leq\beta<t$, and
$d-1=ms+\epsilon$, $0\leq \epsilon<s$. Assume $s\leq t^2-t$ and
$d>d_0$, where
\begin{equation}\label{d0} d_0:=
\begin{cases}
32t^4\quad {\text{if $s\leq 2t-3$,}}\\
8st^4 \quad {\text{if $s\geq 2t-3$ and $\beta < t-\alpha -2$,}}\\
\max(12(s+1)^2,s^3) \quad {\text{if $s\geq 2t-3$ and $\beta \geq
t-\alpha -2$}}.
\end{cases}
\end{equation}
One has:

\medskip
$\bullet$ if either $s\leq 2t-3$ or  $s\geq 2t-3$ and $\beta
<t-\alpha-2$ or $s\geq 2t-3$ and $\beta > t-\alpha-2$ and either
$s-\epsilon-1<\alpha+\beta+2-t$ or $\beta(\alpha+\beta+2-t)\leq
s-\epsilon-1<(\beta+1)(\alpha+\beta+2-t)$, then $g< G(d,s,t)$;

\medskip
$\bullet$ otherwise $g< G(d,s,t)+4t^3$.
\end{theorem}

The proof relies  on the quoted result \cite[Theorem]{P4},  combined
with a purely arithmetic argument, inspired by a remark by
Ellinsgrud and Peskine \cite[p. 2, (B)]{EllP}.

Unfortunately, we are not able to determine the sharp bound.
However, in the range $t^2-2t+3\leq s\leq t^2-t$ (and $d\gg s$), we
are able to give some information on the curves  verifying a flag
condition, with maximal genus, and to determine the sharp bound in
the particular case $d\equiv 0$ (modulo $s$). As in the case
$s>t^2-t$, a hierarchical structure of the family of curves with
maximal genus, verifying flag conditions,  emerges (\cite{P4},
\cite{Duke}, \cite{D}). In fact, we prove the following result (the
number $\beta$ is defined in the claim of of Theorem \ref{main}).

\begin{theorem}\label{main2}
Let $C$ be an irreducible, reduced, non-degenerate curve, of
arithmetic genus $g$ and degree $d$, in the projective space
$\mathbf P^4$ over the complex field. Assume $C$ is not contained in
any hypersurface of degree $<t$ ($t\geq 3$), and in any surface of
degree $<s$ ($s\geq 3$). Assume $t^2-2t+3\leq s\leq t^2-t$, $d>s^4$,
and that $g$ is maximal with respect previous flag condition. Then
one has:
\begin{equation}\label{coarse}
g=G(d,s,t)-\frac{d}{s}\left(\beta-1\right)+O(1),
\end{equation}
where $|O(1)|\leq 2t^3+s^3$. Moreover, $C$ is arithmetically
Cohen-Macaulay, and lies on a flag $S\subset F$, where $S$ is a
surface of degree $s$, $F$ a hypersurface of degree $t$, $S$ is
unique, and its general hyperplane section  is a space extremal
curve not contained in any surface of degree $<t$. In the case
$\epsilon=s-1$ and $d=(m+1)s$, one has
$$
g=G(d,s,t)-\frac{d}{s}\left(\beta-1\right),
$$
and $C$ is the complete intersection of  $S$, with a hypersurface of
degree $m+1$.
\end{theorem}

The proof follows combining various results in Castelnuovo-Halphen
Theory (\cite{GP1}, \cite{GP3}, \cite{Ellia}, \cite{P4}, \cite{D}).
In particular, it relies on the fact that the theory of space
curves, of degree $s$ and not contained in any surface of degree
$<t$, in the range $t^2-2t+3\leq s\leq t^2-t$, is quite similar to
the theory in the classic range $s>t^2-t$ (\cite{GP1},
\cite{GP3},\cite{Ellia}). Moreover, in order to study the sharp case
$d\equiv 0$ (modulo $s$), we use the same argument as in the proof
of \cite[Proposition 15, p. 130]{P4}, taking into account that it
works well also for $s\geq t^2-2t+3$. We have in mind to apply this
analysis also for the remaining cases $0\leq \epsilon < s-1$, in a
forthcoming paper.

As a consequence of previous results, we get the following bound for the speciality
index. Recall that, for a projective integral curve $C\subset
\mathbf P^N$, one defines the {\it speciality index} $e(C)$ of $C$
as the maximal integer $e$ such that $h^0(C,\omega_C(-e))>0$, where
$\omega_C$ denotes the dualizing sheaf of $C$. The number $d_0$,
appearing in the claim, is defined in (\ref{d0}).

\begin{corollary}\label{c1}
Let $C$ be an irreducible, reduced, non-degenerate curve, of degree
$d$, in the projective space $\mathbf P^4$ over the complex field.
Assume $C$ is not contained in any hypersurface of degree $<t$
($t\geq 3$), and in any surface of degree $<s$ ($s\geq 3$). Let
$e(C)$ denote the index of speciality of $C$. Assume $s\leq t^2-t$
and $d>\max(d_0,6st^4)$. Then one has:
\begin{equation}\label{e1}
e(C)\leq \frac{d}{s}+\frac{s}{t}+t-5.
\end{equation}
Moreover, if $d\equiv 0$ (modulo $s$), $s\leq t^2-t$, and $d>s^4$, then one
has:
\begin{equation}\label{e2}
e(C)\leq \frac{d}{s}+2t-7.
\end{equation}
In the range $t^2-2t+3\leq s\leq t^2-t$, the bound (\ref{e2}) is
sharp. In fact, every curve, complete intersection of a surface of
degree $s$, whose general hyperplane section is a space extremal
curve not contained in any surface of degree $<t$, with a
hypersurface of degree $m+1$, attains the bound.
\end{corollary}

\noindent Taking into account that $d\cdot e(C)\leq 2g-2$ ($g:=$ the
arithmetic genus of $C$), previous corollary easily follows from
Theorem \ref{main} and Theorem \ref{main2}. In the range $s>t^2-t$
(and $d>\max(\frac{2}{3}s^4,12(s+1)^2)$), the bound (\ref{e1})  is
already known, and it is sharp \cite[Theorem B, and Remark (i), p.
97]{Speciality}. We do not know whether a curve with maximal
speciality given by (\ref{e2}), is necessarily a complete
intersection as above.

As for the numerical assumptions appearing in previous results,
they are certainly not the sharpest for our purposes. They are
only of the simplest form we were able to conceive.

\bigskip

\section{Notations and preliminaries for Theorem \ref{main}}

In this section we establish some notation, and collect some
numerical results  (i.e. (\ref{stimarho}), Lemma \ref{lemma1},
(\ref{stimaR}), and Lemma \ref{lemma2} below), which we need in the
proof of  Theorem \ref{main}. For the proof of these results,  which
consists in  long and elementary calculations, we refer to the
Appendix at the end of the paper (Section 6).

\medskip
(i) Set $ \tau:=\alpha+1$. Observe that if $s\leq t^2-t$, then
$0\leq \alpha\leq t-2$, $1\leq \tau\leq t-1$, and
\begin{equation}\label{tau}
s>\tau^2-\tau.
\end{equation}
Define $\alpha'$ and $\beta'$ by dividing $s-1=\alpha'\tau+\beta'$,
$0\leq \beta'<\tau$. Let $x$ be the unique integer $0\leq x \leq
t-1$ such that $t-(x+1)(\alpha+1)\leq \beta<t-x(\alpha+1)$. Then we
have $\alpha'=t-1-x$, and $\beta'=\beta -[t-(x+1)(\alpha+1)]$.

\medskip
(ii) As in \cite[p. 120]{P4}, we define the numbers:
$$
G(d,s,t):=\frac{d^2}{2s}+\frac{d}{2}\left[\frac{s}{t}+t-5-\frac{(t-1-\beta)(1+\beta)(t-1)}{st}\right]+\rho+1,
$$
and
$$
G(d,s,\tau):=\frac{d^2}{2s}+\frac{d}{2}
\left[\frac{s}{\tau}+\tau-5-\frac{(\tau-1-\beta')(1+\beta')(\tau-1)}{s\tau}\right]+\rho'+1,
$$
where $\rho=\rho(s,t,\epsilon)$ and  $\rho'=\rho(s,\tau,\epsilon)$
are \lq\lq constant terms\rq\rq, for whose definitions we refer to
the Appendix, (i). When $s\leq t^2-t$, one has
\begin{equation}\label{stimarho}
|\rho|\leq 2t^3,\quad |\rho'|\leq 2t^3.
\end{equation}

\begin{lemma}\label{lemma1}
1) If  $t+1 \leq s\leq t^2-t$ and $\beta < t-\alpha -2$  and $d >
8st^4$, then $G(d,s,\tau)< G(d,s,t)$.

\smallskip
2) If $s\leq t^2-t$ and $\beta >t-\alpha-2$ and either
$s-\epsilon-1<\alpha+\beta+2-t$ or $\beta(\alpha+\beta+2-t)\leq
s-\epsilon-1<(\beta+1)(\alpha+\beta+2-t)$, then $G(d,s,\tau)=
G(d,s,t)$.

\smallskip
3) If either $s\leq t$ or $t+1 \leq s\leq t^2-t$ and $t-\alpha
-2\leq \beta$, then $G(d,s,\tau)=G(d,s,t)+(\rho'-\rho)\leq
G(d,s,t)+4t^3$.
\end{lemma}

\medskip
(iii) We recall also the definition of the bound $G$ for the genus
of a curve in $\mathbf P^4$ of degree $d$, not contained in any
surface of degree $<s$ (compare with \cite[p. 230-231, and p. 241,
Theorem 5.1]{Duke}, and \cite[p. 91-92, (4) and
(4$'$)]{Speciality}). Define $w$ and $w_1$ by dividing $s-1=2w+w_1$,
$0\leq w_1<2$. Then we have
\begin{equation}\label{bduke}
G:=\frac{d^2}{2s}+\frac{d}{2}\left(\frac{2\pi-2}{s}-1\right)+R,
\end{equation}
where $\pi=w(w-1+w_1)$, and $R=R(s,\epsilon)$ is a constant term
(for the definition, see Appendix, (ii)) such that
\begin{equation}\label{stimaR}
|R|\leq s^2.
\end{equation}

\begin{lemma}\label{lemma2}
If $3\leq s\leq 2t-3$ and $d > 32t^4$, then $G<G(d,s,t)$.
\end{lemma}

\bigskip

\section{The proof of  Theorem \ref{main}}

\begin{proof}[Proof of Theorem \ref{main}]
First assume $s\leq 2t-3$. By \cite[p. 241, Theorem 5.1]{Duke}, we
know that  $g\leq G$. Therefore, in this case our claim follows
from Lemma \ref{lemma2}.

\medskip
Next assume $s\geq 2t-3$ and $\beta <t-\alpha-2$. In this case we
may assume also that $t+1\leq s$, otherwise $s=t=3$, and we fall
back in the previous case. If $t+1\leq s$, then $\alpha\geq 1$,
hence $\tau \geq 2$. Moreover,  $\tau<t$. Therefore, $C$ is not
contained in any hypersurface of degree $<\tau$. Since
$s>\tau^2-\tau$ (compare with (\ref{tau})), we may apply
\cite[Theorem]{P4}, and deduce $g\leq G(d,s,\tau)$. By Lemma
\ref{lemma1}, 1), we get $g< G(d,s,t)$.

\medskip
Now assume $s\geq 2t-3$ and $\beta >t-\alpha-2$ and either
$s-\epsilon-1<\alpha+\beta+2-t$ or $\beta(\alpha+\beta+2-t)\leq
s-\epsilon-1<(\beta+1)(\alpha+\beta+2-t)$. As before, we have $g\leq
G(d,s,\tau)$. In this case, Lemma \ref{lemma1}, 2), says that
$G(d,s,\tau)=G(d,s,t)$. Therefore, we get $g\leq G(d,s,t)$. We
cannot have $g=G(d,s,t)$, otherwise, by \cite[Theorem]{P4},  $C$
should be contained in a hypersurface of degree $\tau<t$.

\medskip
Finally, in the remaining cases, as before we have $g\leq
G(d,s,\tau)$. By Lemma \ref{lemma1}, 3), it follows that $g\leq
G(d,s,t)+4t^3$. Again, we cannot have $g=G(d,s,t)+4t^3$, otherwise,
by \cite[Theorem]{P4},  $C$ should be contained in a hypersurface of
degree $\tau<t$.
\end{proof}

\bigskip

\section{Notations and preliminaries for  Theorem \ref{main2} and Corollary \ref{c1}}

In order to prove Theorem \ref{main2} and Corollary \ref{c1}, in
this section we recall some properties of arithmetically
Cohen-Macaulay varieties. Moreover,  we recall the main results of
the theory of space curves of degree $s$, not contained in any
surface of degree $<t$, in the range $t^2-2t+3\leq s\leq t^2-t$ (see
\cite{GP1}, \cite[p. 219]{GP3}, \cite[10.8. Teorema, p. 56]{Ellia}).

\medskip
(i) If $V\subseteq \mathbf P^N$ is an integral projective variety,
we denote by $h_V(t)$ and $p_V(t)$ the Hilbert function and the
Hilbert polynomial of $V$. We denote by $p_a(V)$ its arithmetic
genus. The variety $V$ is said to be {\it arithmetically
Cohen-Macaulay} (shortly a.C.M.) if all the restriction maps
$H^0(\mathbf P^N, \mathcal O_{\mathbf P^N}(i))\to H^0(V,\mathcal
O_V(i))$ ($i\in\mathbf Z$) are surjective, and $H^j(V,\mathcal
O_V(i))=0$ for all $i\in\mathbf Z$ and $1\leq j\leq \dim V-1$
(\cite[9-8]{Sr}, \cite[p. 84]{EH}). If $\dim V\geq 2$, then $V$ is
a.C.M. if and only if its general hyperplane section is. If $V$ is a
curve of degree $d$, and $V'$ denotes its general hyperplane
section, then
$$p_a(V)\leq \sum_{i=1}^{+\infty} d-h_{V'}(i),$$ and equality occurs
if and only if $V$ is a.C.M. \cite[p. 83-84]{EH}. Moreover, using
the diagram in \cite[p. 232]{Duke}, one may prove that if $V$ is an
a.C.M. curve, and $e(V)$ denotes its speciality index, then:
$$
e(V)=\max\left\{n:\, h_{V'}(n)<d\right\}-1.
$$

\medskip
(ii) Fix integers $s,t\geq 3$, with $t^2-2t+3\leq s \leq t^2-t$. Let
$\Sigma \subset \mathbf P^3$ be an integral curve, of degree $s$,
not contained in any surface of degree $< t$, and of maximal
arithmetic genus, that we denote by $P=P(s,t)$ (compare with
\cite[p. 219]{GP3}, \cite[p. 43-49]{GP1}, \cite[10.8. Teorema, p.
56]{Ellia}, and (\ref {bGP3}) below). Set
$\nu:=s-(t^2-2t+3)(=\beta-2)$. $\Sigma$ is an a.C.M. curve,
contained in a surface of degree $t$, with {\it caract\`ere
num\'erique} $(n_0,n_1,\dots,n_{t-1})$ given by:
$$
n_i:= \begin{cases} 2t-3-i\quad {\text{if $0\leq i\leq t-3-\nu$,}}
\\ 2t-2-i\quad {\text{if $t-2-\nu\leq i\leq t-2$,}}
\\t\quad {\text{if $i=t-1$,}}
\end{cases}
$$
(\cite[p. 40, p. 45]{GP1}, \cite[p. 20, and proof of 10.4: Lemma, p.
53]{Ellia}). If $h=h(n)$ denotes the Hilbert function of the general
plane section of $\Sigma$, then one has, for every integer $n$,
$$
h(n)=\sum_{i=0}^{t-1}\left[(n-i+1)_+-(n-n_i+1)_+\right],
$$
where $(x)_+:=\max(0,x)$ \cite[3.7: Lemma, p. 20]{Ellia}. It follows
that:
\begin{equation}\label{fh}
h(n)=\binom{n+2}{2}-2\binom{n+2-t}{2}+\binom{n-2t+4}{2}-\binom{n+1-t}{1}-\binom{n+1-t-\nu}{1}
\end{equation}
for every $n$, where we assume $\binom{a}{b}=0$ if $a<b$. Notice
that
\begin{equation}\label{fh2}
h(2t-5)=s-1, \quad h(2t-4)=s.
\end{equation}
We have:
\begin{equation}\label{bGP3}
p_a(\Sigma)=P(s,t)=\sum_{i=1}^{+\infty}
s-h(i)=t^3-5t^2+(\beta+7)t+\frac{\beta^2-7\beta}{2}-1.
\end{equation}
Comparing with (\ref{acca}) (see below) we get:
$$
H(s,t)=P(s,t)+\beta-1.
$$
Therefore, we may write (compare with (\ref{Gacca})):
\begin{equation}\label{GeP}
G(d,s,t)=
\frac{d^2}{2s}+\frac{d}{2s}(2P-2-s)+\frac{d}{s}(\beta-1)+\rho+1.
\end{equation}
Observe that, if $d=(m+1)s$, then $\rho=0$, and therefore, in this
case, we have:
\begin{equation}\label{2GeP}
G(d,s,t)=
\frac{d^2}{2s}+\frac{d}{2s}(2P-2-s)+\frac{d}{s}(\beta-1)+1.
\end{equation}

\bigskip

\section{The proof of  Theorem \ref{main2} and of Corollary \ref{c1}}

\begin{proof}[Proof of Theorem \ref{main2}] We proceed by several steps.

\medskip
{\it Step 1.} {\it First we prove that $g\geq \frac{d^2}{2s}+\frac{d}{2s}(2P-2-s)-\frac{s^3}{2}$, where
$P=P(s,t)$ is the number defined in Section 4, (ii).}

\smallskip
Let $\Sigma\subset \mathbf P^3$ be an integral curve of degree $s$,
not contained in any surface of degree $<t$, and of maximal genus
$P=P(s,t)$ (compare with Section 4, (ii)). Let $X\subset \mathbf
P^4$ be the cone over $\Sigma$. Since $\Sigma$ is a.C.M., and $d\gg
s$, there exists an integral a.C.M. curve $Y$ on $X$, of degree $d$
\cite[Lemma 18]{P4}. Since $d\gg s$, Bezout's theorem implies that
$Y$ is not contained in surfaces of degree $<s$. Moreover, $Y$ is
not contained in a hypersurface of degree $<t$, otherwise this
hypersurface would contain $X$ because $d\gg s$, and this is not
possible, for $\Sigma$ is not contained in a surface of degree $<t$.
Therefore, $Y$ satisfies the flag condition of type $(s,t)$, hence
$p_a(Y)\leq g$, for $g$ is maximal. On the other hand, since $Y$ is
a.C.M., we have $p_a(Y)=\sum_{i=1}^{+\infty} d-h_{Y'}(i)$ ($Y'$:=
general hyperplane section of $Y$), and so, by \cite[Lemma,
(2.1)]{D},  we have $p_a(Y)\geq
\frac{d^2}{2s}+\frac{d}{2s}(2P-2-s)-\frac{s^3}{2}$, hence $g\geq
\frac{d^2}{2s}+\frac{d}{2s}(2P-2-s)-\frac{s^3}{2}$.

\medskip
{\it Step 2.} {\it Next we prove (\ref{coarse}), and  that $C$ lies
on a flag $S\subset F$, where $S$ is a surface of degree $s$, $F$ a
hypersurface of degree $t$, $S$ is unique, and the general
hyperplane section  of $S$ is a space extremal curve not contained
in any surface of degree $<t$.}

\smallskip
If $C$ were not contained in a surface of degree $s$, then, by
Theorem \ref{main} (and \cite[Theorem]{P4} in the case $s=t^2-t$),
one would have $g\leq G(d,s+1,t)+4t^3$. Since
$G(d,s+1,t)=\frac{d^2}{2(s+1)}+O(d)$, by previous step it would
follow $\frac{d^2}{2s}+O(d)\leq \frac{d^2}{2(s+1)}+O(d)$. This is in
contrast with the assumption $d\gg s$ (a direct and elementary
computation shows that we may assume  $d>s^4$). Therefore, there
exists a surface $S$ of degree $s$ containing $C$. Bezout's theorem
implies this surface is unique, because $d\gg s$. Let $\Sigma$ be
the general hyperplane section of $S$. Since $S$ is not contained in
a hypersurface of degree $<t$, and $s\geq t^2-2t+3$, then, by Roth's
theorem \cite[(C), p. 2]{EllP}, $\Sigma$ is not contained in a
surface of degree $<t$. Therefore, $p_a(\Sigma)\leq P$. On the other
hand, since $C\subset S$, by \cite[Lemma, (2.1)]{D}, we have $g\leq
\frac{d^2}{2s}+\frac{d}{2s}(2p_a(\Sigma)-2-s)+\frac{s^3}{2}$. By
Step 1 we deduce:
$$
\frac{d^2}{2s}+\frac{d}{2s}(2P-2-s)-\frac{s^3}{2}\leq g\leq \frac{d^2}{2s}+\frac{d}{2s}(2p_a(\Sigma)-2-s)+\frac{s^3}{2}.
$$
Since $d>s^4$, and $p_a(\Sigma)\leq P$, it follows that
$p_a(\Sigma)=P$. Taking into account (\ref{stimarho}) and
(\ref{GeP}), and that $p_a(\Sigma)=P$, from previous inequalities we
deduce (\ref{coarse}). Moreover, since $p_a(\Sigma)=P$, $\Sigma$ is
an extremal curve. In particular $\Sigma$ is a.C.M., and lies on a
space surface $F'$ of degree $t$. It follows that also $S$ is
a.C.M., and $F'$ lifts to a hypersurface $F$ of degree $t$
containing $S$.

\medskip
{\it Step 3.} {\it $C$ is a.C.M..}

\smallskip
Let $S$ be the surface of degree $s$ containing $C$, as above. Let
$\Gamma$ and $\Sigma$ be  the general hyperplane sections of $C$ and
$S$. Let  $X\subset \mathbf P^4$ be the cone over $\Sigma$. By
\cite[Lemma 18]{P4}, we know there exists an integral a.C.M. curve
$B$ on $X$, whose general hyperplane section $B'$ has the same
Hilbert function as $\Gamma$. The curve $B$ satisfies the flag
condition of type $(s,t)$, because $\deg B=d\gg s$, and $\Sigma$ is
not contained in a surface of degree $<t$. Therefore, we have
$p_a(B)\leq g$. It follows that
$$
p_a(B)=\sum_{i=1}^{+\infty} d-h_{B'}(i)=\sum_{i=1}^{+\infty}
d-h_{\Gamma}(i)\leq g.
$$
Since in general we have $g\leq \sum_{i=1}^{+\infty}
d-h_{\Gamma}(i)$, we deduce $g= \sum_{i=1}^{+\infty}
d-h_{\Gamma}(i)$. This proves that $C$ is a.C.M.

\medskip
{\it Step 4.} {\it The case $d=(m+1)s$.}

\smallskip
Let $S$ be the surface of degree $s$ containing $C$, as above. Let
$\Sigma$ denote the general hyperplane section of $S$. $\Sigma$ is
an a.C.M. space curve, not contained in any surface of degree $<t$,
with maximal genus $P=P(s,t)$. Let $B=S\cap F_{m+1}$ be a complete
intersection of $S$ with a hypersurface of degree $m+1$, and $B'$
its general hyperplane section. From the exact sequence $0\to
\mathcal O_S(-m-1)\to \mathcal O_S\to \mathcal O_B\to 0$, we get the
following relation between the Hilbert polynomials of $S$ and $B$:
$p_B(t)=p_S(t)-p_S(t-m-1)$. Taking into account that
$p_S(t)=s{\binom{t+1}{2}}+t(1-P)+1+p_a(S)$, we deduce (compare with
(\ref{2GeP})):
$$p_a(B)=\frac{d^2}{2s}+\frac{d}{2s}(2P-2-s)+1=G(d,s,t)-\frac{d}{s}(\beta-1).$$
Passing to the
hyperplane sections, we also have the exact sequence $0\to \mathcal
I_{\Sigma, \mathbf P^3}\to \mathcal I_{B', \mathbf P^3}\to \mathcal
O_{\Sigma}(-m-1)\to 0$. Taking into account that $\Sigma$ is a.C.M.,
we deduce the following relation between the Hilbert functions: for
every integer $n$, $h_{B'}(n)=h_{\Sigma}(n)-h_{\Sigma}(n-m-1)$.

Now we observe that, in order to prove the claim, it suffices to
prove that
\begin{equation}\label{dis}
h_{\Gamma}(n)\geq h_{B'}(n)
\end{equation}
for every integer $n$, where $\Gamma$ denotes the general hyperplane
section of $C$. In fact, since $B$ is a.C.M., from previous
inequality (\ref{dis}), we deduce $g=\sum_{i=1}^{+\infty}
d-h_{\Gamma}(i) \leq \sum_{i=1}^{+\infty} d-h_{B'}(i) =p_a(B)$. On
the other hand, since $B$ satisfies the flag condition of type
$(s,t)$, we also have $p_a(B)\leq g$. Therefore
$g=p_a(B)=G(d,s,t)-\frac{d}{s}(\beta-1)$, and $h_{\Gamma}(n)=
h_{B'}(n)$ for every $n$. In particular, we have $h_{\Gamma}(m+1)=
h_{B'}(m+1)=h_{\Sigma}(m+1)-h_{\Sigma}(0)= h_{\Sigma}(m+1)-1$. This
implies there exists a surface of degree $m+1$ containing $\Gamma$,
and not containing $\Sigma$. Since $C$ is a.C.M., this surface lifts
to a hypersurface $G$ of degree $m+1$, containing $C$ and not
containing $S$.  Since $S$ is Cohen-Macaulay, by degree reasons it
follows that $C$ is equal to the complete intersection  $S\cap G$.

In order to prove (\ref{dis}), we argue as follows. First, we observe that, when $n\leq
m$, we have $h_{\Gamma}(n)= h_{B'}(n)$ because both are equal to
$h_{\Sigma}(n)$, by degree reasons. Hence, we only have to examine
the case $n\geq m+1$. Now, from the equality
$h_{B'}(n)=h_{\Sigma}(n)-h_{\Sigma}(n-m-1)$, passing to the
difference, we get
\begin{equation}\label{acm}
\Delta h_{B'}(n)=h(n)-h(n-m-1)
\end{equation}
for every integer $n$, where the function $h=h(n)$ denotes the
Hilbert function of the general plane section of $\Sigma$ (see
(\ref{fh})). Observe that, since $d\gg s$, we have $h(n)=s$ for
$n\geq m$. Using (\ref{fh}), a direct computation proves that, for
every $n\geq m$, one has
$$
s-h(n-m-1)=F_{a,b}(n),
$$
where $F_{a,b}(n)$ is the function defined in \cite[Definition 13,
p. 129]{P4}, with $a=t-1$, and $b=m+t$. Therefore, we may write
$$h_{B'}(n)=\sum_{i=0}^{m}h(i)+\sum_{i=m+1}^{n} F_{a,b}(i)$$
for every $n$. Then we may prove the inequality (\ref{dis}), i.e.
$$
h_{\Gamma}(n)\geq h_{B'}(n)=\sum_{i=0}^{m}h(i)+\sum_{i=m+1}^{n}
F_{a,b}(i),
$$
with the same argument as in the proof of \cite[Proposition 15, p.
130]{P4}, taking into account that \cite[Proposition 8, p. 127]{P4},
\cite[Proposition 12, p. 129]{P4} and \cite[Lemma 14, p. 130]{P4}
hold true also for $s\geq t^2-2t+3$.  (compare with \cite[Remark 10,
p. 128]{P4}).
\end{proof}

\bigskip

\begin{proof}[Proof of Corollary \ref{c1}] Let $g$ be the arithmetic
genus of $C$. Since $d\cdot e(C)\leq 2g-2$, by Theorem \ref{main} we
deduce
$$
e(C)\leq \frac{2(G(d,s,t)+4t^3)-2}{d},
$$
from which (\ref{e1}) follows, taking into account the definition of
$G(d,s,t)$ (Section 2, $(ii)$), (\ref{stimarho}), and that $d\gg 0$.

If $d=(m+1)s$ and $s<t^2-t$, from (\ref{e1}) we get (\ref{e2}).

The bound (\ref{e2}) holds true also if $s=t^2-t$, because, in this
case, instead of the bound $g\leq G(d,s,t)+4t^3$ given by Theorem
\ref{main}, we may apply the more fine bound $g\leq
G(d,s,t)-\frac{d}{s}\left(\beta-1\right)$ given by Theorem
\ref{main2}.

In the range $t^2-2t+3\leq s\leq t^2-t$, the bound (\ref{e2}) is
sharp. In fact, let $B$ be a complete intersection on a surface as
in the claim, and denote by $B'$ its general hyperplane section.
Then $B$ is a.C.M., and therefore
$$
e(B)=\max\left\{n:\, h_{B'}(n)<d\right\}-1.
$$
Combining with (\ref{acm}) and (\ref{fh2}),  we get
$e(B)=\frac{d}{s}+2t-7$.
\end{proof}

\bigskip

\section{Appendix}

We keep all the notation stated in Section 1 and 2.

\bigskip
(i) The function $\rho=\rho(s,t,\epsilon)$ is defined as follows
(see \cite[p. 120]{P4}{\footnote{In the formula defining $\rho$ in
\cite[p. 120]{P4}, there is a misprint. In fact, in the case
$\epsilon\geq s-(\beta+1)(\alpha+\beta +2-t)$, the factor
$\alpha-\beta-6$ must be replaced  by $t-\beta -3$ (compare
 with \cite[p. 2708]{P5}, line 10 from below).}}).

\smallskip
If $\epsilon\geq s-(\beta+1)(\alpha+\beta +2-t)$, divide
$s-\epsilon-1=u(\alpha+\beta +2-t)+v$, and put:
$$
\rho:=\frac{s-1-\epsilon}{s}\left[\frac{s^2}{2t}+\frac{s}{2}(t-4)-\frac{(t-1-\beta)(1+\beta)(t-1)}{2t}+1\right]
+\frac{1+\epsilon}{2s}\left(s-\epsilon+1\right)
$$
$$
+{\binom{u+v+1}{2}}-\frac{1}{2}(\alpha+\beta)(2v+u\alpha+u\beta-u^2)+\frac{1}{2}u(t-1)(t-\beta-3)-1;
$$

if $\epsilon < s-(\beta+1)(\alpha+\beta +2-t)$, divide
$\epsilon=u(\alpha+\beta +1)+v$, and put:
$$
\rho:=\frac{s-1-\epsilon}{s}\left[\frac{s^2}{2t}+\frac{s}{2}(t-4)-\frac{(t-1-\beta)(1+\beta)(t-1)}{2t}+1\right]
+\frac{1+\epsilon}{2s}(s-\epsilon+1)
$$
$$
-\frac{1}{2}(\alpha+\beta)(t-u-1)(\alpha+t+u-3)+\frac{1}{2}\beta(t-1)(2\alpha-6)-1.
$$

Similarly, we define $\rho'=\rho(s,\tau,\epsilon)$ (compare with
Section 2, (i)).

\bigskip (ii)
The number $R$ appearing in (\ref{bduke}) is defined as follows
(\cite{Duke}, \cite[p. 91-92, (4) and (4$'$)]{Speciality}).

First, define $k$ and $\delta$ by dividing $\epsilon=kw+\delta$,
$0\leq \delta<w$, when $\epsilon < (3-w_1)w$. Otherwise, define $k$
and $\delta$ by dividing $\epsilon+2-w_1=k(w+1)+\delta$, $0\leq
\delta< w+1$. Then we have:
$$
R:=\frac{1+\epsilon}{2s}(s+1-\epsilon-2\pi)+w(\epsilon-\delta)-k{\binom{w+1}{2}}+{\binom{\delta}{2}}.
$$

\bigskip (iii)
{\it Sketch of the proof of (\ref{stimarho})}. We only prove that
$|\rho|\leq 2t^3$ in the case $\epsilon\geq s-(\beta+1)(\alpha+\beta
+2-t)$. The analysis of the case $\epsilon <
s-(\beta+1)(\alpha+\beta +2-t)$, and the proof of the estimate
$|\rho'|\leq 2t^3$, are quite similar, therefore we omit them.

Set:
\begin{equation}\label{acca}
H:=H(s,t):=\frac{s^2}{2t}+\frac{s}{2}(t-4)-\frac{(t-1-\beta)(1+\beta)(t-1)}{2t}+1.
\end{equation}
This number is the coefficient of the term $\frac{s-1-\epsilon}{s}$
appearing in the definition of $\rho$. By the way, notice that, if
$s>t^2-t$, then $H$ is the Halphen's bound for the genus of a space
curve of degree $s$, not contained in any surface of degree $<t$
(\cite[p. 1]{GP1}, \cite[10.8. Teorema, p. 56]{Ellia}). We also
notice we may write:
\begin{equation}\label{Gacca}
G(d,s,t)=\frac{d^2}{2s}+\frac{d}{2s}\left(2H-2-s\right)+\rho+1.
\end{equation}
Taking into account that $s-1=\alpha t+\beta$, we may rewrite $H$:
\begin{equation}\label{2acca}
2H=\alpha t^2+(\alpha^2-4\alpha)t+\beta^2+2\alpha+(2\alpha-1)\beta.
\end{equation}
The function $\alpha\to \alpha^2-4\alpha$ is growing for $\alpha
\geq 2$. Therefore, when $\alpha\geq 2$, since $0\leq \alpha\leq
t-2$ and $0\leq \beta\leq t-1$, it follows that:
\begin{equation}\label{manner}
0\leq 2H\leq
(t-2)t^2+((t-2)^2-4(t-2))t+(t-1)^2+\quad\quad\quad\quad\quad\quad\quad\quad
\end{equation}
$$
\quad\quad\quad\quad\quad\quad+2(t-2)+(2(t-2)-1)(t-1)
=2\left(t^3-\frac{7}{2}t^2+\frac{5}{2}t+1\right).
$$
This inequality holds true also for $\alpha\leq 1$. Hence
\begin{equation}\label{H}
0\leq H\leq t^3-\frac{7}{2}t^2+\frac{5}{2}t+1.
\end{equation}
Since $\frac{s-1-\epsilon}{s}\geq 0$, $H\geq 0$, and $t-\beta-3\geq
-2$, from the definition of $\rho$ we deduce:
$$
\rho\geq -\frac{1}{2}(\alpha+\beta)(2v+u\alpha+u\beta)-u(t-1)-1.
$$
Taking into account that
\begin{equation}\label{rho1bis}
s\leq t^2-t,\quad \alpha\leq t-2,\quad v\leq \alpha+\beta+1-t\leq
\beta-1\leq t-2, \quad u\leq \beta\leq t-1,
\end{equation}
substituting in a similar manner as in
(\ref{manner}), it follows that
\begin{equation}\label{rho1}
\rho\geq -2t^3+5t^2-\frac{3}{2}t-\frac{7}{2}\geq -2t^3.
\end{equation}
Moreover, since  $\frac{s-1-\epsilon}{s}\leq 1$,
$$
\frac{1+\epsilon}{2s}(s-\epsilon+1)\leq \frac{1}{2}(s+1)\leq
\frac{1}{2}(t^2-t+1),
$$
and $u\leq \beta$ (which implies that $u\beta-u^2\geq 0$, so
$-\frac{1}{2}(\alpha+\beta)(2v+u\alpha+u\beta-u^2)\leq 0$), from
(\ref{H}), (\ref{rho1bis}), and the definition of $\rho$, it follows
that:
$$
\rho\leq H +
\frac{1}{2}(t^2-t+1)+\frac{1}{2}(u+v+1)(u+v)+\frac{1}{2}ut(t-1)\leq
$$
$$
\leq \left(t^3-\frac{7}{2}t^2+\frac{5}{2}t+1\right) +
\frac{1}{2}(t^2-t+1)+\frac{1}{2}(2t-2)(2t-3)+\frac{1}{2}t(t-1)^2=
$$
$$
=\frac{3}{2}t^3-2t^2-\frac{5}{2}t+\frac{9}{2}\leq 2t^3.
$$
Combining this estimate with (\ref{rho1}), we deduce $|\rho|\leq
2t^3$, in the case $s\leq t^2-t$ and $\epsilon\geq
s-(\beta+1)(\alpha+\beta +2-t)$.

\bigskip (iv)
{\it Proof of (\ref{stimaR})}. Recall that $s-1=2w+w_1$, $0\leq
w_1\leq 1$, and $\pi=w(w-1+w_1)$ (compare with Section 2, (iii), and
with this Appendix, (ii)). Hence we have:
$$
s+1-\epsilon-2\pi\leq s+1-2\pi=\frac{1}{2}(-s^2+6s+w_1^2-2w_1-1)\leq
\frac{1}{2}(6s-s^2).
$$
Therefore, if $s\geq 6$, then $s+1-\epsilon-2\pi\leq 0$. In this
case, taking into account that $w\leq (s-1)/2$ and that $\delta\leq
w$, we have:
$$
R\leq w\epsilon+\frac{1}{2}\delta(\delta-1)\leq
w(s-1)+\frac{1}{2}w(w-1)\leq \quad \quad\quad\quad\quad
$$
$$
\quad\quad\quad\quad\leq
\frac{1}{2}(s-1)^2+\frac{1}{2}\frac{s-1}{2}\frac{s-3}{2}\leq
\frac{5}{8}(s-1)^2\leq s^2.
$$
An easy direct computation shows that the inequality $R\leq s^2$
holds true also when $3\leq s\leq 5$. Therefore we have:
\begin{equation}\label{pR}
R\leq s^2.
\end{equation}
On the other hand we have:
$$
R\geq
\frac{1+\epsilon}{2s}(-\epsilon-2\pi)-w\delta-\frac{1}{2}w(w+1)k.
$$
Hence:
$$
-R\leq \frac{1}{2}(s-1+2\pi)+w\delta+\frac{1}{2}w(w+1)k=
$$
\begin{equation}\label{sR}
=\frac{1}{2}\left(\frac{1}{2}s^2-s+\frac{1}{2}+w_1-\frac{1}{2}w_1^2\right)+w\delta+\frac{1}{2}w(w+1)k.
\end{equation}
When $\epsilon < w(3-w_1)$, then $\delta\leq w-1$ and $k\leq 2$.
Therefore, in this case, from (\ref{sR}) we have:
$$
-R\leq \frac{1}{2}\left(\frac{1}{2}s^2-s+1\right)+w(w-1)+w(w+1)=
$$
$$
\frac{1}{2}\left(\frac{1}{2}s^2-s+1\right)+2w^2\leq
\frac{1}{2}\left(\frac{1}{2}s^2-s+1\right)+2\frac{(s-1)^2}{4}=
\frac{3}{4}s^2-\frac{3}{2}s+1\leq s^2.
$$
When $\epsilon \geq w(3-w_1)$, then $\delta\leq w$ and $k\leq
\frac{2(s+1)}{s}$. From  (\ref{sR}) we get:
$$
-R\leq
\frac{1}{2}\left(\frac{1}{2}s^2-s+1\right)+w^2+w(w+1)\frac{s+1}{s}\leq
$$
$$
\frac{1}{2}\left(\frac{1}{2}s^2-s+1\right)+\frac{(s-1)^2}{4}+(s^2-1)\frac{s+1}{4s}=
\frac{1}{4s}\left(3s^3-3s^2+2s-1\right)\leq s^2.
$$
Combining with (\ref{pR}), we get $|R|\leq s^2$.

\bigskip (v)
{\it Proof of  Lemma \ref{lemma1}}. Consider the coefficient of
$\frac{d}{2}$ in the expression defining $G(d,s,t)$ and
$G(d,s,\tau)$ (Section 2, (ii)):
$$
A:=\frac{s}{t}+t-5-\frac{(t-1-\beta)(1+\beta)(t-1)}{st},
$$
$$
A':=\frac{s}{\tau}+\tau-5-\frac{(\tau-1-\beta')(1+\beta')(\tau-1)}{s\tau}.
$$
We have:
\begin{equation}\label{G}
G(d,s,\tau)=G(d,s,t)+\frac{d}{2}(A'-A)+(\rho'-\rho).
\end{equation}
Observe that (compare with (\ref{acca})):
$$
H=\frac{s}{2}(A+1)+1, \quad H'=\frac{s}{2}(A'+1)+1,
$$
where $H'=H(s,\tau)$. Hence, by (\ref{2acca}) (compare with Section
2, (i)), we have:
$$
A'-A=\frac{2}{s}(H'-H)=\frac{1}{s}\left[\alpha'\tau^2+(\alpha'^2-4\alpha')\tau+\beta'^2
+2\alpha'+(2\alpha'-1)\beta'\right]
$$
$$
-\frac{1}{s}\left[\left(\alpha t^2+(\alpha^2-4\alpha)t+\beta^2
+2\alpha+(2\alpha-1)\beta\right) \right].
$$
Simplifying, we get:
$$
A'-A=\frac{\alpha x}{s}(\alpha x+\alpha+x-2t+3+2\beta).
$$
Hence, if $\alpha x=0$, then $A'=A$. When $\alpha>0$ and $x>0$,
since $-(t-\beta)<-x(\alpha+1)$, we have:
$$
\alpha x+\alpha+x-2t+3+2\beta=\alpha x+\alpha+x-2(t-\beta)+3
$$
$$
\leq \alpha x+\alpha+x-2(x(\alpha+1)+1)+3=(1-x)(\alpha+1)\leq 0.
$$
If $x=1$, then the number
$$
\alpha x+\alpha+x-2t+3+2\beta=2\alpha+4-2t+2\beta
$$
vanishes if and only if $\beta=t-\alpha-2$. Summing up, we get: {\it
in any case, one has  $A'\leq A$. Moreover,  $A'=A$ if and only if
either $\alpha=0$ or $x=0$ or $x=1$ and $\beta=t-\alpha-2$, i.e. if
and only if either $s\leq t$ or $s\geq t+1$ and $t-\alpha-2\leq
\beta <t$. In particular, when $A'<A$, then $A-A'\geq
\frac{\alpha}{s}\geq \frac{1}{st}$}.

We deduce the following.

1) If $t+1\leq s\leq t^2-t$ and $\beta <t-\alpha-2$, then $A'<A$.
Therefore, from (\ref{stimarho}) and (\ref{G}), we deduce that
$G(d,s,\tau) < G(d,s,t)$ for $d>8st^4$. In fact, in this case, we
have $\frac{d}{2}(A'-A)+(\rho'-\rho)<0$, because $
\frac{2(\rho'-\rho)}{A-A'}\leq 2\cdot 4t^3\cdot st$.

\smallskip
2) If $s\leq t^2-t$ and $t-\alpha-2<\beta$, then $A=A'$. Hence,
(\ref{G}) becomes $G(d,s,\tau)=G(d,s,t)+(\rho'-\rho)$. A direct
computation, which we omit,  shows that, in this case, if either
$s-\epsilon-1<\alpha+\beta+2-t$ or $\beta(\alpha+\beta+2-t)\leq
s-\epsilon-1<(\beta+1)(\alpha+\beta+2-t)$, then $\rho=\rho'$. Hence,
we have $G(d,s,\tau)=G(d,s,t)$.

\smallskip
3) If either $s\leq t$ or $t+1 \leq s\leq t^2-t$ and $t-\alpha
-2\leq \beta$, then $A=A'$. Therefore, by (\ref{stimarho}) and
(\ref{G}), we get $G(d,s,\tau)=G(d,s,t)+(\rho'-\rho)\leq
G(d,s,t)+4t^3$.

\medskip
This concludes the proof of Lemma \ref{lemma1}.

\bigskip (vi)
{\it Proof of  Lemma \ref{lemma2}}. Consider the coefficient of
$\frac{d}{2}$ in the formula (\ref{bduke}) defining $G$:
$$
A'':=\frac{2\pi-2}{s}-1.
$$
We have:
\begin{equation}\label{GG}
G=G(d,s,t)+\frac{d}{2}(A''-A)+(R-\rho-1).
\end{equation}
A direct computation proves that:
$$
A''-A=\frac{1}{2s}\left[(\alpha^2-2\alpha)t^2+(2\alpha\beta+6\alpha-2\alpha^2)t+
(-4\alpha-\beta^2-4\alpha\beta+w_1)\right].
$$
If $\alpha=0$, i.e. $s\leq t$, then $s=\beta+1$, and
$$
A''-A=\frac{1}{2s}(-\beta^2+w_1).
$$
If $t+1\leq s\leq 2t-3$, then $\alpha=1$, and we have:
$$
A''-A=\frac{1}{2s}\left[-(t-\beta-2)^2+w_1\right].
$$
In both cases we have $A''-A<-\frac{1}{2s}$. Therefore, from
(\ref{GG}) and (\ref{stimaR}), we deduce that $G < G(d,s,t)$ for $d>
32t^4$, because in this case $\frac{d}{2}(A''-A)+(R-\rho-1)<0$ (in
fact: $ \frac{2(R-\rho-1)}{A-A''}\leq 4s(s^2+2t^3)\leq 32t^4$).

This concludes the proof of Lemma \ref{lemma2}.

\bigskip
\begin{remark}\label{notaconclusiva}
(i) A similar argument shows that if $2t-2\leq s\leq 2t$, then
$A''=A$, and that if $t>2$ ed $s\geq 2t+1$, then $A''>A$.

\smallskip
(ii) When $s\geq t+1$ and $t-\alpha-2\leq \beta$, it may happen that
$G(d,s,\tau)>G(d,s,t)$. For instance, if $s=t^2-2t+6$ and
$\epsilon=s-25$, then $\rho'-\rho=2(t+1)$.
\end{remark}

\bigskip
\bigskip
{\it Acknowledgment.} I would like to thank Luca Chiantini and Ciro
Ciliberto for valuable discussions and suggestions, and their
encouragement.

\end{document}